\documentclass[12pt,leqno,a4paper]{amsart}

\usepackage{amssymb,amsmath,amsthm,amsfonts,enumerate}   
\usepackage{hyperref}
\usepackage{fancyvrb}
\usepackage[all]{xy}    
\usepackage{comment}

\hypersetup{pdftex,                           
bookmarks=true,
pdffitwindow=true,
colorlinks=true,
citecolor=black,
filecolor=black,
linkcolor=black,
urlcolor=blue,
hypertexnames=true}

\newcommand{\IZ}{{\mathbb{Z}}}

\newcommand{\fA}{{\mathfrak{A}}}     

\DeclareMathOperator{\Out}{Out}                  
\DeclareMathOperator{\Ext}{Ext}                   %
\DeclareMathOperator{\Syl}{Syl}                  
\DeclareMathOperator{\Inf}{Inf}                    

\DeclareMathOperator{\soc}{soc}			

\DeclareMathOperator{\GL}{GL}

\DeclareMathOperator{\PSL}{PSL}
\DeclareMathOperator{\SL}{SL}

\DeclareMathOperator{\PSU}{PSU}
\DeclareMathOperator{\PSp}{PSp}

\let\lra=\longrightarrow

\newtheorem{thm}{Theorem}[section]
\newtheorem{lem}[thm]{Lemma}

\newtheorem{cor}[thm]{Corollary}
\newtheorem{prop}[thm]{Proposition}

\theoremstyle{theorem}
\newtheorem{thma}{Theorem}         
\newtheorem{cora}[thma]{Corollary}         
\newtheorem{quesa}[thma]{Question}

\theoremstyle{definition}

\newtheorem{hyp}[thm]{Hypothesis}

\theoremstyle{remark}
\newtheorem{rem}[thm]{Remark}

\VerbatimFootnotes

\begin{document}


\title{Simple modules in the Auslander-Reiten quiver of principal blocks with abelian defect groups}

\date{\today}

\author{Shigeo Koshitani}
\address{Department of Mathematics, Graduate School of Science,
Chiba University, 1-33 Yayoi-cho, Inage-ku, Chiba, 263-8522, Japan.}
\email{koshitan@math.s.chiba-u.ac.jp}

\author{Caroline Lassueur}
\address{Caroline Lassueur\\ FB Mathematik, TU Kaiserslautern, Postfach 3049, 67653 Kaisers\-lautern, 
Germany.}
\email{lassueur@mathematik.uni-kl.de}

\thanks{The authors gratefully acknowledge financial support by the funding scheme TU Nachwuchsring 
Individual Funding 2016 granted to the second author by the TU Kaiserslautern.
The first author was also supported by  the Japan Society for
Promotion of Science (JSPS), Grant-in-Aid for Scientific Research
(C)15K04776, 2015--2018.}

\begin{abstract}
Given an odd prime $p$, we investigate the position of simple modules in the stable Auslander-Reiten quiver 
of the principal block of a finite group with non-cyclic abelian Sylow $p$-subgroups. In particular, we prove a 
reduction to finite simple groups. In the case that the characteristic is $3$, we prove that simple modules in 
the principal block all lie at the end of their components. \\
\end{abstract}


\maketitle


\pagestyle{myheadings}
\markboth{S. Koshitani \& C. Lassueur}{Simple modules in the AR-quiver of principal blocks}



\section{Introduction} \label{sec:intro}
The position of simple modules in the stable Auslander-Reiten quiver of the group algebra $kG$ over a field 
$k$ of characteristic $p$  of a finite group $G$ of order divisible by $p$ is a question that was partially 
investigated in the 1980's and the 1990's in a series of articles by different authors. We refer the reader in 
particular to \cite{Kaw97, KMU00, KMU01, BU01} and the references therein. The aim of this note is to come 
back to the following question:

\begin{quesa}\label{Ques:Q1}
Let $B$ be a wild $p$-block of $kG$.  Under which conditions do all simple $B$-modules  lie at the end of 
their connected components in the stable Auslander-Reiten quiver of~$kG$?
\end{quesa}

A main reason of interest in this question lies in the fact that  a simple $kG$-module lies at the end of its 
component if and only if the heart of its projective cover is indecomposable.\par
In this article, we focus attention on the case in which the principal block $B_0(kG)$ is of wild representation 
type with abelian defect groups and the prime $p$ is odd, which amounts to requiring that the $p$-rank of 
$G$ is at least $2$.
The case $p=2$ was treated by Kawata-Michler-Uno in \cite[Theorem 5]{KMU00}. We aim at  extending their 
results and part of their methods to arbitrary primes. Further, we note that the cases when $B_0(kG)$ is of 
finite or tame representation type are well-understood. In the former case, the distance of a simple module to 
the rim of its connected component (a tube of shape $(\IZ/\!e\IZ)A_m$) is a function of its position in the 
Brauer tree of the block, while in the later case the position of the simple modules in their connected 
components is given by Erdmann's work on tame blocks \cite{ERDMANNbook}.\\

Assuming the field $k$ is algebraically closed we prove the following results:

\begin{thma}\label{thm:solvablequotient}\label{MainTheorem}
Let $G$ be a finite group and $N\trianglelefteq G$
a normal subgroup such that $G/N$ is  solvable of $p'$-order.
Let $B$ and $b$ be wild blocks of $kG$ and $kN$ respectively such that $1_B=1_b$. 
If every simple $b$-module lies at the end of its connected component in the stable Auslander-Reiten quiver 
of $kN$, then
every  simple $B$-module lies at the end of its connected component in the stable Auslander-Reiten quiver of 
$kG$.
\end{thma}

\begin{thma}\label{thm:redsimples}
Let $p$ be an odd prime.  Let $G$ be a finite group with non-cyclic abelian Sylow $p$-subgroups and $O_{p'}(G)=1$. Write $O^{p'}(G)=Q \times H_1\times \cdots \times H_m$ ($m\geq 0$), where $Q$ is an abelian $p$-group and  $H_i$ is a non-abelian finite simple group with non-trivial Sylow $p$-subgroups for each $1\leq i\leq m$. Assume that one of the following conditions is satisfied:
\begin{enumerate}
\item[\rm(i)] $Q\neq 1$; or
\item[\rm(ii)] $Q=1$ and $m\geq 2$; or
\item[\rm(iii)] $Q=1$, $m=1$ and every simple $B_0(kH_1)$-module lies at the end of its connected 
component in the stable Auslander-Reiten quiver of $kH_1$.
\end{enumerate}
Then  every simple $B_0(kG)$-module lies at the end of its connected component in the stable Auslander-Reiten quiver of $kG$.
\end{thma}

\begin{cora}\label{thm:corredsimples}
Let $p$ be an odd prime. Assume that every  simple $B_0(kH)$-module lies at the end of its connected 
component in the stable Auslander-Reiten quiver of $kH$ for every non-abelian finite simple group $H$ with 
non-cyclic abelian Sylow $p$-subgroups.  Then every simple $B_0(kG)$-module lies at the end of its 
connected component in the stable Auslander-Reiten quiver of $kG$ for any finite group $G$ with non-cyclic 
abelian Sylow $p$-subgroups.
\end{cora}

We note that if $p=2$, then the analogues of Theorem~\ref{thm:redsimples} and 
Corollary~\ref{thm:corredsimples} were essentially proven by Kawata-Michler-Uno \cite{KMU00}, although not 
stated in these terms.
As a corollary, we also obtain the equivalent of \cite[Theorem 5(a)]{KMU00} for the prime $3$.

\begin{thma}\label{thm:p=3}
Assume $p=3$. Let $G$ be a finite group with abelian Sylow $3$-subgroups. If $B_0(kG)$ is a wild $3$-
block, then every simple $B_0(kG)$-module lies at the end of its connected component in the stable 
Auslander-Reiten quiver of $kG$.
\end{thma}

The paper is organised as follows: in Section~2, we recall the state of knowledge on the subject and extend a 
result of Kawata's \cite[Theorem 1.5]{Kaw97} to describe more precisely the indecomposable summands of 
the heart of the projective cover of a simple module not lying on the rim of its component. In Section~3, we 
consider groups having a solvable quotient of $p'$-order and prove Theorem~\ref{thm:solvablequotient}. In 
Sections~4 and~5, we proceed to a reduction of Question~\ref{Ques:Q1} for principal blocks to the case of 
finite non-abelian simple groups and prove Theorem~\ref{thm:redsimples} and 
Corollary~\ref{thm:corredsimples}. Finally in Section~6 we deal with the case $p=3$ and prove 
Theorem~\ref{thm:p=3}.


\section{Notation and preliminary results} \label{sec:nota} 

Throughout, we assume that $k$ is an algebraically closed field of characteristic $p>0$, and we let $G$ 
denote a finite group of order divisible $p$. 
All modules are assumed to be finitely generated right modules. For a $p$-block $B$, we write $1_B$ for the 
corresponding block idempotent and ${\mathrm{IBr}}(B)$ for the set of isomorphism classes of simple $kB$-
modules. Furthermore, unless otherwise specified, we assume $B_0:=B_0(kG)$, the principal block of $kG$, 
is wild. (When the defect groups of $B$ are abelian, we may therefore assume that  a Sylow $p$-subgroup of 
$G$ is non-cyclic, or equivalently that the $p$-rank of $G$ is at least $2$).
We denote by $k_G$ the trivial $kG$-module.
\par
We let $J:=J(kG)$ denote  the Jacobson radical of $kG$.  For a $kG$-module $U$, we 
define $J(kG)^0:=kG$  and for any integer 
$i\geq 0$, $\mathrm{soc}^i(U):=\{u\in U\,|\,u\,J^i=\{0\}  \}$, 
then inductively for any $i\geq 1$, we write
\begin{align*}
L_i(U):= U\,J^{i-1}/U\,J^i \text{ and }
S_i(U):= {\mathrm{soc}}^i(U)/{\mathrm{soc}}^{i-1}(U)\, 
\end{align*}
for the $i$-th {\it Loewy layer}  and the $i$-th {\it socle layer} of $U$, respectively.
We use throughout without further mention the following well-known properties:

\begin{lem}Assume $N\trianglelefteq G$ of index prime to $p$.
\begin{enumerate}
  \item[\rm(a)] We have  $J=\tilde J\,kG=kG\,\tilde J$ where  $\tilde J:=J(kN)$. \label{radical}
  \item[\rm(b)] \label{LoewySocle}
Let  $X$ be a $kG$-module and $Y$ a $kN$-module, then for any $i\geq 1$ we have
$$L_i(X){\downarrow}_N=L_i(X{\downarrow}_N)\text{ and }
  S_i(X){\downarrow}_N=S_i(X{\downarrow}_N)$$
 and
$$L_i(Y){\uparrow}^G=L_i(Y{\uparrow}^G) \text{ and } S_i(Y){\uparrow}^G=S_i(Y{\uparrow}^G)\,.$$
\end{enumerate}
\end{lem}

\begin{proof}
Part (a)  is a well-known result of Villamayor \cite{Vill59} and part (b) follows from~(a).   
\end{proof}

Given a $kG$-module $M$, we denote  by $\Omega^n(M)$ $(n\in\IZ)$ its $n$-th Heller translate.  Given a 
simple $kG$-module $S$, we denote by $P(S)$ its projective cover and by $\mathcal{H}(P(S))$ the heart of 
$P(S)$, that is  
$\mathcal{H}(P(S))=P(S)J/\soc(P(S))$.
\par
We let $\Gamma_s(kG)$, resp. $\Gamma_s(B)$,  denote the stable Auslander-Reiten quiver of $kG$, resp. 
of the $p$-block $B$, and denote by $\Gamma_s(M)$ the connected component of $\Gamma_s(kG)$ 
containing a given indecomposable $kG$-module $M$. Moreover, by convention, we use the terminology  
\emph{AR-component} to refer to a connected component of $\Gamma_s(kG)$.\\

Erdmann \cite{Erd95} proved that all components of the stable Auslander-Reiten quiver belonging to a wild 
block have tree class $A_{\infty}$, that is are of the form $\IZ A_{\infty}$ or infinite tubes $\IZ A_{\infty}/\langle \tau^a\rangle$ of rank $a$, where $\tau=\Omega^2$ is the Auslander-Reiten shift.
\par
In a component with tree class $A_\infty$  an indecomposable non-projective $kG$-module $M$ is said to lie 
\emph{at the end (or on the rim)} of its AR-component if the projective-free part of the middle term $X_M$ of 
the Auslander-Reiten sequence
$$\mathcal{A}(M): 0\lra \Omega^2(M)\lra X_M\lra M\lra 0$$
terminating at $M$ is indecomposable. For a non-projective simple $kG$-module $S$, the Auslander-Reiten 
sequence terminating at $\Omega^{-1}(S)$ is of the form
$$\mathcal{A}(\Omega^{-1}(S)): 0\lra \Omega(S)\lra \mathcal{H}(P(S))\oplus P(S)\lra \Omega^{-1}(S)\lra 0$$
and is called the \emph{standard sequence associated to $S$}.  In this set up, clearly a simple module $S$ 
lies at the end of its component if and only if $\mathcal{H}(P(S))$ is indecomposable, and $S$ lies in a tube if 
and only if $S$ is periodic (i.e. $\Omega$-periodic).\\

We also recall that for a selfinjective algebra the shape of the components of the stable Auslander-Reiten 
quiver
is an invariant of its Morita equivalence class. By the above, the property of lying on the rim of its AR-component for a non-projective simple module is also invariant under Morita equivalence.\par

Simple $kG$-modules are known to lie on the rim of their $AR$-components in the following cases:

\begin{thm}\label{thm:end}
Let $B$ be a wild $p$-block of $kG$. Then every simple $B$-module lies at the end of its  AR-component in 
all of the following cases:
\begin{enumerate}
  \item[\rm(a)] $G$ has a non-trivial normal $p$-subgroup 
  (\cite[Theorem 2.1]{Kaw97}); 
  \item[\rm(b)] $G$ is $p$-solvable (\cite[Corollary 2.2]{Kaw97});
  \item[\rm(c)] $G$ is a perfect finite group of Lie type in the defining characteristic and $B$ has full defect 
(\cite[Theorem]{KMU01});
  \item[\rm(d)] $G$ has an abelian Sylow $2$-subgroup and $B$ is 
the principal $2$-block. \cite[Theorem~5]{KMU00};
  \item[\rm(e)] $G$ is a symmetric group, an alternating group or a Schur cover of the latter groups, and the 
defect of $B$ is divisible by $p^3$ (\cite[\S 5]{BU01}). 
  \end{enumerate}
\end{thm}

\noindent Moreover, we will use the following computational criterion throughout:

\begin{thm}[{}{Kawata's Criterion on Cartan matrices \cite[Theorem 1.5]{Kaw97}}]\label{thm:kawCartan}
Let $B$ be a wild $p$-block of $kG$. Suppose that there exists a simple $B$-module $S$ lying on the $n$-th 
row from the end of  $\Gamma_s(S)$, where $n\geq 2$ is minimal with this property. Then there exist 
pairwise non-isomorphic simple $B$-modules $S_2,\ldots, S_n$ with the following properties:
\begin{enumerate}
  \item[\rm(a)]  For each $2\leq i \leq n$$, S_i\cong \Omega^{2(i-2)}(S_2)$ and   lies at the end of $\Gamma_s(\Omega(S))$;
  \item[\rm(b)]  The projective covers of $P(S_i)$ of the simple modules 
  $S_i \ (2\leq i\leq n)$ are uniserial of length $n+1$ 
  with the following Loewy structure:

$$
{
P(S_2)=\boxed{
\begin{matrix}  S_2\\ S_3\\ \vdots \\ S_n \\ S \\ S_2
\end{matrix} 
             }\,,
\ 
P(S_3)=\boxed{
\begin{matrix} S_3\\ \vdots \\ S_n \\ S \\ S_2 \\S_3
\end{matrix} 
             }\,, \ 
\cdots, \ 
P(S_n)=\boxed{
\begin{matrix} S_n\\ S\\S_2\\ \vdots \\ \ \ \,S_{n-1} \\ S_n
\end{matrix} 
             }\,.   
}                  
$$
The Cartan matrix of $B$ is given by
$$\left(\begin{matrix}2 & 1 & \cdots & \cdots & 1 & 0 & \cdots & 0 \\
1 & 2 & \ddots &  & \vdots & \vdots &  & \vdots \\
\vdots & \ddots & \ddots & \ddots & \vdots & \vdots &  & \vdots \\
\vdots &  & \ddots & 2 & 1 & 0 & \cdots & 0 \\
1 & \cdots & \cdots & 1 & \ast & \cdots & \cdots & \ast \\
0 & \cdots &  \cdots & 0 & \vdots &  &  & \vdots \\
\vdots &  &  & \vdots & \vdots &  &  & \vdots \\
0 & \cdots & \cdots & 0 & \ast & \cdots & \cdots & \ast
\end{matrix}\right)\,,$$
where the columns are  labelled by $S_n,\ldots, S_2,S,\ldots$ in this order.
\end{enumerate}
\end{thm}

\begin{rem}
\begin{enumerate}
  \item[\rm(a)] If the Cartan matrix of a block has the shape of Theorem~\ref{thm:kawCartan}(b) above with 
$n=2$, then the simple module $S$ corresponding to the second column lies on the 2nd row of its AR-component. Indeed in this case 
  $P(S_2)=\boxed{\begin{smallmatrix}S_2\\ S\\S_2\end{smallmatrix}}$ and the standard sequence associated 
to $S_2$ is
  $$0\lra \Omega(S_2)\lra S\oplus P(S_2)\lra \Omega^{-1}(S_2)\lra0\,,$$
  so that $S_2$ lies at the end of its AR-component and $S$ on the 2nd row of its AR-component.\\
  A converse to Kawata's  Criterion need  not  be true in general for an $n\geq 3$.
  \item[\rm(b)]  The above was used to produce two counter-examples of simple modules not lying at the end 
of their AR-components. Namely, the group $F_4(2)$ for $p=5$ has a simple module in the principal block of 
dimension 875823 lying on the the 2nd row of its AR-component, and the group $2.Ru$ for $p=3$ has a 
faithful simple module also  lying on the 2nd row. See \cite[\S4]{KMU01}. Both counter-examples are obtained 
thanks to the decomposition matrices of these groups computed by G.~Hi\ss.
\end{enumerate}

\end{rem}

We can now improve Kawata's result by describing more accurately the structure of the heart of the projective 
cover of the simple module $S$ lying on the $n$-th row of its AR-component.

\begin{cor}\label{cor:toKawataCriterion}
With the assumptions and the notation of Theorem~\ref{thm:kawCartan}, we have that the heart of the 
projective cover of the simple module $S$ is decomposable and has a uniserial indecomposable summand of 
length $n-1$. More precisely 
\begin{align*}
\mathcal H(P(S)) \ &= \ 
\boxed{\begin{matrix} S_2\\ S_3 \\ \vdots \\ S_n \end{matrix} }
\, \bigoplus \, V\,,
\\
\end{align*}
where $V\,{\not=}\,\{0\}$  is an indecomposable $kG$-module.
\end{cor}

\begin{proof}
Let $S_2,\ldots, S_n$ be the simple $kG$-modules given by Theorem~\ref{thm:kawCartan} and set $\mathcal{H}:=\mathcal H(P(S))$. 
By the structure of $P(S_2)$ (see Theorem~\ref{thm:kawCartan}(b)), there exists a uniserial 
$kG$-module 
$$ W\,:=\, \boxed{\begin{matrix} S_2\\S_3\\ \vdots\\S_n\\ S \ \,
         \end{matrix} }\,.
$$
We may assume that
$W\subseteq P(S)$ since ${\mathrm{soc}}(W)=S$, and hence $W\subseteq P(S)J$
since $S_2\,{\not\cong}\,S$. This implies that 
$W/S\subseteq P(S)J/S=\mathcal H$. Namely,
there exists a $kG$-submodule $U$ of $\mathcal H$ with
$$
\mathcal H\supseteq U\,=\,
\boxed{\begin{matrix} S_2\\S_3\\ \vdots\\S_n
         \end{matrix} }\,.
$$
On the other hand, by the structure of $P(S_n)$ (see Theorem~\ref{thm:kawCartan}(b))  
there is a $kG$-epimorphism
$$
P(S)\ \twoheadrightarrow 
\boxed{ \begin{matrix}
       S \ \\ S_2 \\ S_3 \\ \vdots \\ S_n
        \end{matrix}        }
\text{\ \ and hence a $kG$-epimorphism }
\psi\, : \, 
\mathcal H\,=\,P(S)J/S\ \twoheadrightarrow 
\boxed{ \begin{matrix}
       S_2 \\ S_3 \\ \vdots \\ S_n
        \end{matrix}        }
$$
since $S_n\,{\not\cong}\,S$.
Set $V:=\ker(\psi)$. Since the entry $(1, n)$ of the Cartan matrix  is equal to one, then by definition of $\psi$,
and by Theorem~\ref{thm:kawCartan}(b) we know that
${\mathrm{soc}}(U)\cap{\mathrm{soc}}(V)=\{0\}$, hence 
$U\cap V=\{0\}$. Thus we have a direct sum
$U\oplus V\subseteq\mathcal H$.
Then, by the definitions of $U$ and $V$ and by counting the number of
composition factors of $\mathcal H$, we obtain
$U\oplus V=\mathcal H$. Clearly $V\,{\not=}\,\{0\}$ by assumption.
\end{proof}

Finally given a normal subgroup $H\trianglelefteq G$ and $\tilde{b}$ a $p$-block of $H$,  we will use the 
group $G[\tilde{b}]$ defined by Dade \cite{Dade73}.
\begin{lem}[Dade]\label{G[b]}
Let $H\trianglelefteq G$ such that 
$p\, {\not|}\,|G/H|$, and let $P$ be a Sylow $p$-subgroup of $H$.
Let $\tilde{b}:=B_0(kH)$ and $B:=B_0(kG)$ be the principal blocks of $kH$ and $kG$,
respectively. Set $N:=H\,C_G(P)$. Then the following hold:
\begin{enumerate}
\item[\rm (a)]  The block $\tilde{b}$ is $G$-invariant. 
the Frattini argument, that is ''Sylow's theorem'').
\item[\rm (b)]  $N=G[\tilde{b}]$ and $N\trianglelefteq G$.
\item[\rm (c)] If $b$ denotes the principal block of $kN$, then $1_B=1_b$.  
\end{enumerate}
\end{lem}

\begin{proof}

(a) Obvious since $\tilde{b}$ is the principal block.

(b) This follows from \cite[Corollary 12.6]{Dade73}
since $\tilde{b}$ is the principal block  (see \cite[Proof of Lemma~3]{Dade77}).
As $\tilde{b}$ is $G$-invariant, the fact that $G[\tilde{b}]\trianglelefteq G$ follows from 
\cite[Proposition 2.17]{Dade73}.

(c) (see also \cite[p.303 line 10]{Kue95}) As $\tilde{b}$ is $G$-invariant,  $1_{\tilde{b}}$ is an idempotent of 
$Z(kG)$ and we can write 
\begin{equation*}\label{blockIdemp}
1_{\tilde{b}}=1_B  + 1_{B_1} + \cdots  + 1_{B_n}
\end{equation*}
for an integer $n\geq 0$ and for distinct non-principal blocks $B_1, \ldots, B_n$ of $kG$. Thus, $1_{\tilde{b}}\,1_B=1_B$. Namely,
$$
1_B \ \in 1_{\tilde{b}}\,Z(kG)\subseteq 1_{\tilde{b}}\,C_{kG}(H)=:C.
$$
This implies $1_B\in Z(C)$ since $1_B\in Z(kG)$.
Hence it follows from \cite[Corollary 4]{Kue90} and part (b) 
that 
\begin{equation*}
\begin{split}
1_B\in C[\tilde{b}] & =Z(\tilde{b})*G[\tilde{b}]\quad\quad\quad \quad  \text{ (where $*$ denotes the crossed 
product) } \\
                             & \subseteq Z(kH)*N  \subseteq kN\,.
\end{split}
\end{equation*}
Thus $1_B\in Z(kN)$. On the other hand, since $b$ is the principal block of $kN$, 
we have $1_{b}$ is $G$-invariant, so that 
$1_{b}\in Z(kG)$. Hence as above we can write
$$
1_{b}=1_B+1_{B'_1}+\cdots +1_{B'_t}$$
for an integer $t\geq 0$ and for some
distinct non-principal blocks $B'_1, \cdots, B'_t$ of $kG$. 
Set $\tilde e:=1_{b}-1_B\in Z(kN)$
since $1_B\in Z(kN)$.
Therefore $1_{b}=1_B+\tilde e$  is a decomposition of~$1_b$ 
into orthogonal idempotents of $Z(kN)$. This implies that $\tilde e=0$,
and hence $1_{b}=1_B$.
\end{proof}

\vspace{2mm}
\section{Groups having a solvable quotient of $p'$-order} \label{sec:}

\begin{hyp}\label{hyp1}\label{hypo2}
Assume  that:
\begin{enumerate}
  \item[\rm(a)] $G$ is a finite group of order divisible by $p$ and $N\trianglelefteq G$ is a normal subgroup 
such that $|G/N|=:q$ is a prime number with $q\neq p$,
and we set $G/N=:\langle gN \rangle$ for an element $g\in G\backslash N$. 
  \item[\rm(b)]  $B$ and $b$ are wild blocks of $kG$ and $kN$ respectively such that $1_B=1_b$.
\end{enumerate}
\end{hyp}

\begin{lem}\label{simples} \label{extension}
Assume Hypothesis~\ref{hyp1} holds. Let $\zeta\in k^\times$ be a primitive $q$-th root of unity in $k$,
and for each $1\leq j\leq q$ let $Z_j$ be the one-dimensional $k(G/N)$-module defined by
$Z_j:=\langle\alpha_j\rangle_k$ and 
$\alpha_j{\cdot}gN := \zeta^{j-1}\alpha_j$,
so that in particular $Z_1=k_{G/N}$. The following holds:
\begin{enumerate}
  \item[\rm(a)] If $\mathcal S\in{\mathrm{IBr}}(B)$ is such that
$\mathcal S{\downarrow}_N$ is not simple, then for each $1\leq j\leq q$, 
$$
\mathcal S\otimes_k Z_j \cong \mathcal S
$$
as $kG$-modules, where we see $Z_j$ as a $kG$-module via inflation.
   \item[\rm(b)]   There are integers $m\geq 1$ and $\ell\geq 0$ such that
\begin{align*}
{\mathrm{IBr}}(B)&=
\{ \mathcal S_{ij}\, |\,1\leq i\leq m; 1\leq j\leq q \}\, \bigsqcup\,
\{\mathcal S_i\,| \,m+1\leq i\leq m+\ell\} \text{ and }
\\
{\mathrm{IBr}}(b)&=
\{\mathcal T_i \,| \,1\leq i\leq m \}\,\bigsqcup\,
\{\mathcal T_{ij}\, |\, m+1\leq i\leq m+\ell, 1\leq j\leq q\}\,,
\end{align*}
where  for each $1\leq i\leq m$ and each $1\leq j\leq q$,
$$\mathcal S_{ij}{\downarrow}_N =\mathcal T_i\quad\text{ and }
\quad \mathcal T_i{\uparrow}^G = \mathcal S_{i1}\oplus\cdots\oplus\mathcal S_{iq}\,,
$$  and for each $m+1\leq i\leq m+\ell$ and each $1\leq j\leq q$,
$$\mathcal S_i{\downarrow}_N 
= \mathcal T_{i1}\oplus\mathcal T_{i2}\oplus\cdots \oplus\mathcal T_{iq}\quad\text{ and }\quad 
\mathcal T_{ij}{\uparrow}^G = \mathcal S_i$$
where we may assume that
$\mathcal T_{ij} := {\mathcal T_{i1}}^{g^{j-1}}$.\\
Moreover, we can assume that for each $1\leq j\leq q$,
$$ 
\mathcal S_{ij} = \mathcal S_{i1}\otimes_k Z_j\,.
$$
\end{enumerate}
\end{lem}

\begin{proof}
(a) Let $1\leq j\leq q$. By  assumption and Clifford's theory we have that 
\begin{align*}
(\mathcal S\otimes_kZ_j){\downarrow}_N 
&= \mathcal S{\downarrow}_N\otimes_k k_N \cong\mathcal S{\downarrow}_N.
\\
&= \mathcal T\oplus\mathcal T^g\oplus\cdots\oplus 
   \mathcal T^{g^{q-1}}
\end{align*}
for some $\mathcal T\in{\mathrm{IBr}}(b)$.
Hence $\mathcal T{\uparrow}^G\cong\mathcal S$,
and $\mathcal T{\uparrow}^G\cong\mathcal S\otimes_kZ_j$ for each $1\leq j\leq q$.

(b)  As by Hypothesis \ref{hyp1} the quotient $G/N$ is cyclic, the claim follows from the result of Schur-Clifford 
\cite[Chap.~3 Corollary 5.9 and Problem 11(i)]{NagaoTsushima}.
\end{proof}

\begin{lem}\label{projCover}
Assume Hypothesis~\ref{hyp1} holds. Let $\mathcal S\in{\mathrm{IBr}}(B)$.
\begin{enumerate}
\item[\rm (a)]
If $\mathcal S{\downarrow}_N=:\mathcal T$ is simple, then
$P(\mathcal S){\downarrow}_N\cong P(\mathcal T)$
and
$\mathcal H(P(\mathcal S)){\downarrow}_N \cong
 \mathcal H(P(\mathcal T))$.
\item[\rm (b)]
If $\mathcal S{\downarrow}_N$ is not simple, then we can write
$\mathcal S{\downarrow}_N= 
\mathcal T_1\oplus \mathcal T_2\oplus\cdots\oplus
\mathcal T_q$ with 
$\mathcal T_j:={\mathcal T_1}^{g^{j-1}}$ for each $1\leq j\leq q$ and  we have that
$$
P(\mathcal S){\downarrow}_N\cong P(\mathcal T_1)
\oplus\cdots\oplus P(\mathcal T_q)
\quad\text{ and }\quad\mathcal H(P(\mathcal S)){\downarrow}_N\cong \bigoplus_{j=1}^q \mathcal 
H(P(\mathcal T_j)).
$$
\end{enumerate}
\end{lem}

\begin{proof}(a) Obviously
\begin{align*}
\mathcal T&=\mathcal S{\downarrow}_N
=(P(\mathcal S)/P(\mathcal S)J){\downarrow}_N
  = P(\mathcal S){\downarrow}_N 
    / (P(\mathcal S)J){\downarrow}_N
\\
&=P(\mathcal S){\downarrow}_N / (P(\mathcal S)\,kG\tilde J)
\text{ \ by Lemma~\ref{radical} }
\\
&= P(\mathcal S){\downarrow}_N
/ P(\mathcal S){\downarrow}_N\,\tilde J.
\end{align*}
Hence the top of $P(\mathcal S){\downarrow}_N$ is 
$\mathcal T$, which implies that
$P(\mathcal S){\downarrow}_N\cong P(\mathcal T)$.
Therefore,
$$
\mathcal H(P(\mathcal S)){\downarrow}_N
= (P(\mathcal S)\,J/\mathcal S){\downarrow}_N 
= \mathcal H(P(\mathcal T)).
$$
(b) Similar to (a). 
\end{proof}

\begin{prop}\label{hyp_b}
Assume Hypothesis~\ref{hyp1} holds. If every simple module $ T\in{\mathrm{IBr}}(b)$ lies at the end of its 
AR-component, then every simple 
module $ S\in{\mathrm{IBr}}(B)$ lies at the end of its AR-component.
\end{prop}

\begin{proof}
Let $ S\in{\mathrm{IBr}}(B)$ be a simple module.
First assume that $ S{\downarrow}_N=:T\in{\mathrm{IBr}}(b)$ is  simple.  Then by Lemma \ref{projCover}(a)
$$\mathcal H(P(S)){\downarrow}_N\cong\mathcal H(P(T))\,.$$
But by assumption $\mathcal H(P(T))$ is indecomposable, therefore so is $\mathcal H(P(S))$.\par

We assume now for the rest of the proof that $S{\downarrow}_N$ is not simple. If $S$ 
lies at the end of its AR-component, then there is nothing to do. Therefore we now 
also assume that $S$ lies on the $n$-th row from the bottom of $\Gamma_s(S)$ for 
an integer $n\geq 2$, minimal (as in Kawata's Criterion on Cartan matrices). 
By Lemma~\ref{simples}(b), 
$$S{\downarrow}_N = T_{11}\oplus\cdots\oplus T_{1q}\quad \text{ and}\quad {T_{1j}}
{\uparrow}^G=S\quad \text{ for each } \,1\leq j\leq q\,,$$ 
where $T_{1j}:={T_{11}}^{g^{j-1}}$ for $1\leq j\leq q$  are non-isomorphic simple 
modules in ${\mathrm{IBr}}(b)$. We also set $T_1:=T_{11}$. \par 

Let $S_2, \ldots, S_n$ 
be the simple modules given by Theorem~\ref{thm:kawCartan}.\\

\noindent \textsc{Claim 1.} If the modules $S_2{\downarrow}_N,\ldots, S_n{\downarrow}_N$ are all non-simple, then we have a contradiction.\\

\noindent \textit{Proof of Claim 1.}
By assumption and Lemma~\ref{simples}, we can write
$$S_i{\downarrow}_N = T_{i1}\oplus T_{i2}\oplus\cdots\oplus T_{iq}\,.$$
For each $2\leq i\leq n$ we define $T_i\in{\mathrm{IBr}}(b)$ by $T_{ij}={T_i}^{g^{j-1}}$, where $1\leq j\leq q$. 
We claim that 
   \begin{equation*}\label{P(T_i)}
P(T_2)=\boxed{
\begin{matrix} T_2\\ T_3\\ \vdots \\ T_n \\ T_1 \\ T_2
\end{matrix} 
             }\,,
\ 
P(T_3)=\boxed{
\begin{matrix} T_3\\ \vdots \\ T_n \\ T_1 \\ T_2\\ T_3
\end{matrix} 
             }\,, \ 
\cdots, \ 
P(T_n)=\boxed{
\begin{matrix} T_n\\ T_1\\T_2\\ \vdots \\  \ \ T_{n-1} \\ T_n
\end{matrix} 
             }.
  \end{equation*}
Indeed, we know by Theorem~\ref{thm:kawCartan}(b) and Lemma~\ref{projCover}(b) 
that
\begin{align*}
P(T_2)\oplus P(T_2)^g\oplus\cdots\oplus P(T_2)^{g^{q-1}}
&=
P(S_2){\downarrow}_N
\\
&=\boxed{
\begin{matrix} S_2\\ S_3\\ \vdots \\ S_n \\ S \\ S_2
\end{matrix}
}   {\downarrow}_N
=
\boxed{
\begin{matrix}
T_2 & {T_2}^{g}&\cdots\, {T_2}^{g^{q-1}} \\
T_3 & {T_3}^{g}&\cdots\, {T_3}^{g^{q-1}} \\
  &  \cdots       \\
T_n & {T_n}^{g}&\cdots\, {T_n}^{g^{q-1}} \\    
T_1 & {T_1}^{g}&\cdots\, {T_1}^{g^{q-1}} \\
T_2 & {T_2}^{g}&\cdots\, {T_2}^{g^{q-1}} 
\end{matrix}   }\,,
\end{align*}
where the boxes mean  the Loewy and socle series of the $kN$-modules.
Since the left-hand-side is a direct sum of exactly $q$ indecomposable
$kN$-modules that are $\langle g\rangle$-conjugate to each other,
by interchanging the indices of $T_3, \ldots, T_n, T_1$, we may assume that
the PIM $P(T_2)$ has the desired structure.  Then automatically 
the structures of $P(T_3),\ldots, P(T_n)$ are as claimed.\par 
Now, using a similar argument as above, we also obtain 

$$P(T_1)\oplus P(T_1)^g\oplus\cdots\oplus P(T_1)^{g^{q-1}}
= P(S){\downarrow}_N =
\boxed{ \begin{matrix}
    \boxed{T_1 \oplus {T_{1}}^{g}\oplus\cdots\oplus  {T_{1}}^{g^{q-1}} }
                \\  
                \\
     \boxed{
         \boxed{ \begin{matrix}  
                        S_2 \\
                        S_3 \\
                        \vdots \\
                        S_n 
                  \end{matrix}  }{\downarrow}_N
 \ \bigoplus \
V{\downarrow}_N
         }          
\\   
\\
\ \boxed{
       T_1\oplus {T_1}^{g}\oplus\cdots\oplus {T_1}^{g^{q-1}} 
        } 
\end{matrix}
}\,,
$$
where the last equality holds by Corollary~\ref{cor:toKawataCriterion}. Hence we have
\begin{align*}
\mathcal H(P(T_1))\oplus \mathcal H(P(T_1))^g\oplus\cdots\oplus
\mathcal H(P(T_1))^{g^{q-1}} &= 
\boxed{
\begin{matrix}
T_2\oplus {T_2}^g\oplus\cdots\oplus{T_2}^{g^{q-1}} \\
T_3\oplus {T_3}^g\oplus\cdots\oplus{T_3}^{g^{q-1}} \\
\vdots                                                   \\
T_n\oplus {T_n}^g\oplus\cdots\oplus{T_n}^{g^{q-1}}    
\end{matrix}
             } \ \oplus \ V{\downarrow}_N 
\\ 
&= 
\begin{pmatrix}
\boxed{
\begin{matrix} T_2\\T_3\\ \vdots\\ T_n\end{matrix} }
\oplus
\boxed{
\begin{matrix} T_2\\T_3\\ \vdots\\ T_n\end{matrix} }^g 
\oplus \cdots \oplus
 \boxed{
\begin{matrix} T_2\\T_3\\ \vdots\\ T_n\end{matrix} }^{g^{q-1}}
\end{pmatrix}
\ \oplus \ V{\downarrow}_N 
\end{align*}
since $P(T_2),\ldots,P(T_n)$ are uniserial by the above.

But we are assuming that $T_2,\ldots,T_n$ lie at the end of their AR-components, so that 
$\mathcal H(P(T_2)), \ldots, \mathcal H(P(T_n))$ are indecomposable. Therefore the right-hand side term in 
the later equation has exactly $q$
indecomposable direct summands. This implies that $V=\{0\}$, hence a contradiction.\\


\noindent \textsc{Claim 2.} If the modules $S_2{\downarrow}_N, \ldots, S_n{\downarrow}_N$ are all simple, 
then we have a contradiction.\\

\noindent \textit{Proof of Claim 2.}
Set $T_i:=S_i{\downarrow}_N$ for $2\leq i\leq n$. We have 
$$
S{\downarrow}_N=T_1\oplus {T_1}^g\oplus\cdots\oplus {T_1}^{g^{q-1}}.
$$

By the assumption and Lemma \ref{simples}, for each $2\leq i\leq n$ we can write 
${T_i}\!\uparrow^G=S_{i1}\oplus \cdots\oplus S_{iq}$
with $S_{ij}:=S_{i1}\otimes_kZ_j$ for $1\leq j\leq q$. In particular $S_{i1}=S_i$ for each $2\leq i\leq n$. By 
Theorem~\ref{thm:kawCartan}(b)
$$
P(S_2) = 
\boxed{
\begin{matrix} S_2\\ S_3\\ \vdots \\ S_n\\ S \ \, \\S_2
\end{matrix}
          }\,,
$$ 
so Lemma \ref{extension}(a) implies that  
$$
P(S_{2j})=P(S_2)\otimes_k Z_j 
=\boxed{
\begin{matrix} S_{2 j}\\ S_{3 j}\\ \vdots \\ S_{n j}\\ S \ \ \\S_{2 j}
\end{matrix}   } 
\text{ \ \ for } 1\leq j\leq q.
$$
This yields that there exists a $kG$-module
\begin{equation*}
\boxed{
\begin{matrix}
\boxed{
\begin{matrix} S_2\\ S_3 \\ \vdots \\ S_n
\end{matrix}
}
\oplus
\boxed{
\begin{matrix} S_{2 2}\\ S_{3 2} \\ \vdots \\ S_{n 2}
\end{matrix}
}
\oplus\cdots\oplus
\boxed{
\begin{matrix} S_{2 q}\\ S_{3 q} \\ \vdots \\ S_{n q}
\end{matrix}
}
\\
\\
S
\end{matrix}
}=:W
\end{equation*}
with  simple socle isomorphic to $S$.
Therefore $W/S$ has a proper uniserial submodule
$$U:=\boxed{
\begin{matrix} S_2\\ S_3 \\ \vdots \\ S_n
\end{matrix}
}\,.$$
Now by Corollary~\ref{cor:toKawataCriterion},  $U|\mathcal H(P(S))$, so that by Lemma~\ref{extension}(a)
\begin{align*}
\boxed{
\begin{matrix} S_{2 j}\\ S_{3 j} \\ \vdots \\ S_{n j}
\end{matrix}
}
\ = \ 
\boxed{
\begin{matrix} S_2\\ S_3 \\ \vdots \\ S_n
\end{matrix}
      }  \otimes_k Z_j
\ &= \ 
(U\otimes_k Z_j)\ {\Big|}\ (\mathcal H(P(S))\otimes_k Z_j)
\cong\mathcal H(P(S\otimes_k Z_j))
\cong\mathcal H(P(S)) 
\end{align*}
for each $1\leq j\leq q$. Therefore $q=2$ since $\mathcal H(P(S))$ has exactly two
non-projective indecomposable direct summands  by the assumption that $S$  does not lie at the end of its 
AR-component.
Notice that this already provides a contradiction in case the characteristic of $k$ is $2$, since we assume 
$q\neq p$. So we now assume that $p\geq 3$.  Then, the Loewy and socle structures of PIMs $P(S)$, 
$P(S_i)$ and $P(S_{i2})$ for $2\leq i\leq n$  are:
$$
\boxed{
\begin{matrix}
S
\\
   \boxed{ \begin{matrix}S_2\\ S_3\\ \vdots\\ S_n \end{matrix} }
   \ \ 
   \boxed{ \begin{matrix}S_{2 2}\\ S_{3 2}\\ \vdots\\ S_{n 2} \end{matrix} }
\\
S   
\end{matrix}  }
, \ \   
\boxed{ \begin{matrix}S_2\\S_3\\ \vdots \\ S_n\\ S\\ S_2  \end{matrix}  }
, \ \
\boxed{ \begin{matrix}S_{2 2}\\S_{3 2}\\ \vdots \\ S_{n 2}\\ S
                          \\ S_{2 2}  \end{matrix}  }
                       , \ \ 
 \boxed{ \begin{matrix}S_3\\ \vdots \\ S_n \\ S\\ S_2\\ S_3  \end{matrix}  }
, \ \ 
\boxed{ \begin{matrix}S_{3 2}\\ \vdots \\ S_{n 2}\\ S \\ S_{2 2}\\S_{3 2}
            \end{matrix}  }
                       ,   \ \  
                          \cdots
                        , \   \
\boxed{ \begin{matrix}S_n\\S\\ S_2\\ \vdots \\ \ \ S_{n-1} \\ S_n  \end{matrix}  }
     , \ \
\boxed{ \begin{matrix}S_{n 2}\\S \ \\ S_{2 2}\\ \vdots \\ \ \ S_{n-1, 2} \\ S_{n,2}       
             \end{matrix}  }.
$$
Now considering the restrictions $S{\downarrow}_N$ and $S_i{\downarrow}_N$ for $2\leq i\leq n$, we obtain 
by Lemma~\ref{projCover} that  the Loewy and socle structures of the PIMs  $P(T_1)$, $P(T_1^g)$ and 
$P(T_i)$ for each $2~\leq~i~\leq~n$ are
$$
\boxed{\begin{matrix}
T_1\\T_2\\T_3 \\ \vdots\\ T_n\\ T_1
\end{matrix} }    ,  \ \ 
\boxed{\begin{matrix}
{T_1}^g\\T_2\\ T_3\\ \vdots\\ T_n\\ {T_1}^g
\end{matrix} }    ,  \ \ 
\boxed{\begin{matrix}
T_2\\T_3\\ \vdots\\ T_n\\ T_1\oplus{T_1}^g \\ T_2
\end{matrix} }    ,  \ \ 
\boxed{\begin{matrix}
T_3\\ \vdots\\ T_n\\ T_1\oplus{T_1}^g \\ T_2 \\ T_3
\end{matrix} }    ,  \ \  \cdots , \ \ 
\boxed{\begin{matrix}
T_n\\ T_1\oplus{T_1}^g \\ T_2\\ T_3 \\ \vdots \\ T_n
\end{matrix} }
$$
since $T_1 \,{\not\cong}\,{T_1}^g$.
Now, as the dimension of any PIM for $kN$ is divisible by $|N|_p=:p^a$ for an integer $a\geq 1$,
and since $\dim\,T_1=\dim\,{T_1}^g$, we  have for each $2\leq i\leq n$
\begin{align*}
0 &\equiv \dim\, P(T_i)-\dim\,P(T_{1})=\dim\,T_i \text{ (mod }p^a)\,,
\end{align*}

so that
$$
0\equiv \dim\,P(T_1)
 \equiv  \dim\,P(T_1)-(\dim\,T_2+\dim\,T_3+\cdots +\dim\,T_n) 
 = 2\cdot\dim\,T_1 \text{ (mod }p^a).
 $$
This implies that
$$\dim\,T_1 \equiv 0   \text{ (mod }p^a)$$
since $p\,{\not=}\,2$ (since $q=2$). 
Thus, $\dim\,T_i\equiv 0$ (mod $p^a$) for any $1\leq i\leq n$.
Now, looking at the composition factors of PIMs 
$P(T_1), P({T_1}^g), P(T_2),\ldots,P(T_n)$,
we know that
${\mathrm{IBr}}(b)=\{ T_1, {T_1}^g, T_2, \ldots, T_n\}$, which implies that
$p^a\,|\,\dim\, \mathcal T$ for any $\mathcal T\in{\mathrm{IBr}}(b)$.
Now it follows from Brauer's result
\cite[Chap.3, Theorem 6.25]{NagaoTsushima} that there is a simple 
$\mathcal T\in{\mathrm{IBr}}(b)$ such that 
$\nu_p(\dim\,\mathcal T)=a-d(b)$ (where $d(b)$ is the defect of $b$). Hence we have a contradiction
since $b$ is a wild block, i.e. of positive defect.\\

\noindent \textsc{Claim 3.}  \begin{enumerate}
    \item[\rm (a)]
If there is an integer $2\leq m\leq n-1$ such that $S_2{\downarrow}_N, \ldots, S_m{\downarrow}_N$ are not 
simple and
$S_{m+1}{\downarrow}_N$ is simple, then we have a contradiction.    
    \item[\rm (b)]
If there is an integer $2\leq m\leq n-1$ such that
$S_2{\downarrow}_N, \ldots, S_m{\downarrow}_N$ are simple and
$S_{m+1}{\downarrow}_N$ is not simple, then we have a contradiction.  
\end{enumerate}
\medskip 

\noindent \textit{Proof of Claim 3.}
(a) Set $T_{m+1}:=S_{m+1}{\downarrow}_N$.
By Lemma~\ref{simples} there exists a simple module
$T_m\in{\mathrm{IBr}}(b)$ with
$S_m{\downarrow}_N=T_m\oplus{T_m}^g\oplus\cdots\oplus{T_m}^{g^{q-1}}$. Then, by 
Lemma~\ref{simples},
$$T_{m+1}{\uparrow}^G =
   S_{m+1}\oplus S_{m+1,2}\oplus\cdots\oplus S_{m+1,q}$$
 where $S_{m+1, j}:=S_{m+1}\otimes_k Z_{j}$ for each $1\leq j\leq q$ and $T_m{\uparrow}^G = S_m$. 
 By the structure of $P(S)$, we have that $\Ext^1_{kG}(S_m,S_{m+1})\neq 0$. Therefore by Eckmann-
Shapiro's Lemma we have that $\Ext^1_{kN}(T_m,T_{m+1})\neq 0$. Thus there exists a $kN$-module with 
Loewy structure 
$$
\boxed{  \begin{matrix}T_m\\ T_{m+1} \end{matrix} }\,.$$
So it follows from Lemma~\ref{radical} that
$$
\boxed{  \begin{matrix}T_m\\ T_{m+1} \end{matrix} }{\uparrow}^G
\ = \ 
\boxed{ \begin{matrix} S_m \\ 
  S_{m+1}\oplus S_{m+1,2}\oplus\cdots\oplus S_{m+1, q}
            \end{matrix}  }
$$           
where the right-hand side box 
is the Loewy and socle series. But $P(S_m)$ is uniserial by 
Theorem~\ref{thm:kawCartan}(b), so applying again Lemma~\ref{radical}, we must have $q=1$, which 
contradicts the assumption that $q$ is prime.\\
(b) follows in a similar  fashion using a dual argument.\\

Altogether, Claims 1-3 prove that the simple modules $S_2,\ldots, S_n$ cannot exist, therefore $S$ must lie 
at the end of its AR-component. 
\end{proof}

As a consequence of the above discussion we obtain Theorem~\ref{thm:solvablequotient} of the Introduction.

\begin{proof}[Proof of Theorem~\ref{thm:solvablequotient}]
Because $G/N$ is solvable of order prime to $p$, it follows by induction on $|G/N|$, that we may assume that 
$|G/N|$ is a prime distinct from $p$. Then Proposition~\ref{hyp_b} yields the result.
\end{proof}

\section{The principal block of $O^{p'}(G)$} \label{sec:}

From now on, we assume that $p\geq 3$ and $G$ is a finite group with non-cyclic abelian Sylow $p$-subgroups. Because we consider the principal block only, we assume that ${O_{p'}(G)=1}$. The structure 
of~$O^{p'}(G)$ is given by the following well-known result of Fong-Harris \cite{FH}.

\begin{lem}[{}{\cite[5A--5C]{FH}}]\label{lem:FH}
Let $p$ be an odd prime. Let $G$ be a finite group with a non-trivial abelian Sylow $p$-subgroup. Then
$$O^{p'}(G/O_{p'}(G))\cong Q \times H_1\times \cdots \times H_m\,,$$
where  $m$ is a non-negative integer (i.e. possibly $O^{p'}(G/O_{p'}(G))\cong Q$), $Q$ is an abelian $p$-group, and for each $1\leq i\leq m$, $H_i$ is a non-abelian simple group with non-trivial Sylow $p$-subgroups.
\end{lem}

Therefore, we fix the notation $O^{p'}(G)=Q \times H_1\times \cdots \times H_m$, where $Q$ is an abelian  
$p$-group, and $H_1,\ldots, H_m$ are non-abelian simple groups with non-trivial Sylow $p$-subgroups as 
given by the lemma.

\vspace{2mm}
\subsection{Simple modules in infinite tubes $\IZ A_{\infty}/\langle \tau^a\rangle$}

\begin{lem}[{}{\cite[Lemma 5.2]{KMU00} generalised version}]\label{lem:periodicProduct}
Let $H=\tilde{H}_1\times \cdots \times \tilde{H}_m$ $(m\geq 1)$ be a finite group such that $p\mid |\tilde{H}_i|$ for each $1\leq i\leq m$. If $B_0(kH)$ is a wild block and contains a periodic simple module, then $m=1$.
\end{lem}

\begin{proof}
Let $S$ be a simple periodic $B_0(kH)$-module. Then  we may write $S=S_1\otimes_k \cdots \otimes_k\, S_m$ where $S_i$ is a simple $B_0(k\tilde{H}_i)$-module for each $1\leq i\leq m$. Then, by iterating 
\cite[Lemma 2.2]{KMU00}, there exists an index $1\leq i_o\leq m$ such that $S_{i_0}$ is periodic and $S_j$ 
is a projective $k\tilde{H}_j$-module for each $1\leq j\neq i_o\leq m$. But $B_0(k\tilde{H}_j)$ cannot contain a 
simple projective module,  since we assume that $p\mid |\tilde{H}_i|$ for each $1\leq i\leq m$. Hence this 
forces $H=\tilde{H}_{i_0}$, i.e. $m=1$.
\end{proof}

As a consequence, the existence of simple periodic modules in the principal block  lying in tubes drastically 
restricts  the possible structure of $O^{p'}(G)$.

\begin{cor}\label{cor:tubesO^p'}
If $B_0(kG)$ contains a periodic simple module, then $O^{p'}(G)=H_1$ is a non-abelian finite simple group 
with non-cyclic abelian Sylow $p$-subgroups.
\end{cor}

\begin{proof}
By Lemma~\ref{lem:periodicProduct}, either $O^{p'}(G)=Q$ or $O^{p'}(G)=H_1$. But the former cannot 
happen. Indeed, the indecomposable direct summands of the restriction to  $O^{p'}(G)$ of a simple periodic 
$kG$-module are all simple periodic modules, however the unique simple $kQ$-module is the trivial module, 
which is not periodic since we assume that $B_0(kG)$ is wild, and hence $Q$ is non-cyclic. The leaves only 
the possibility $O^{p'}(G)=H_1$, and the $p$-rank of $H_1$ must be at least $2$ again because we assume 
that $B_0(kG)$ is wild.
\end{proof}

This immediately leads to the following reduction to non-abelian simple groups:

\begin{cor}\label{cor:periodic}
Assume that every  periodic simple $B_0(kH)$-module lies at the end of its AR-component  for every non-abelian finite simple group $H$ with non-cyclic abelian Sylow $p$-subgroups. Then every simple periodic 
$B_0(kO^{p'}(G))$-module  lies at the end of its AR-component  for any finite group $G$ with $O_{p'}(G)=1$ 
and non-cyclic abelian Sylow $p$-subgroups.
\end{cor}

\subsection{Simple modules in $\IZ A_{\infty}$-components}

\begin{lem}\label{lem:nonperiodicProduct}
Let $H=\tilde{H}_1\times \cdots \times \tilde{H}_m$ $(m\geq 1)$ be a finite group with abelian Sylow $p$-subgroups 
such that $p\mid |\tilde{H}_i|$ for each $1\leq i\leq m$. If $B_0(kH)$ is a wild block containing a non-periodic 
simple module $S$ not lying at the end of its AR-component, then $m=1$.
\end{lem}

This lemma and its proof below generalises parts of the proof of \cite[Theorem~5(i)]{KMU00}.

\begin{proof}
Assume that $m\geq 2$.
Then by Theorem~\ref{thm:kawCartan}(b), there exists a simple $B_0(kH)$-module $T$ lying at the end of $\Gamma_s(\Omega(S))$.
By Kn\"orr's Theorem \cite[3.7 Corollary]{Knorr79}, we know that the vertices of the simple modules in 
$B_0(kH)$ are the Sylow $p$-subgroups of $H$, because they are abelian. Now by assumption $
\Gamma_s(S)\cong  \IZ A_{\infty}$, which implies that all the modules in $\Gamma_s(S)$ and $
\Gamma_s(\Omega(S))$ have the Sylow $p$-subgroups as their vertices by \cite[Theorem]{OkUn94}. 
So all the modules in $\Gamma_s(S)$ and $\Gamma_s(\Omega(S))$ are not projective relatively to the 
subgroup $N:=\tilde{H}_1\times \cdots \times \tilde{H}_{m-1}$  as it does not contain a Sylow $p$-subgroup of 
$H$. Thus, as $p\neq 2$, all the simple direct summands of $S{\downarrow}_N$ belong to blocks of  defect 
zero by \cite[Lemma~1.4]{KMU00}.
But 
$$B_0(kH)=B_0(kN)\otimes_k B_0(k \tilde{H}_{m})$$
 and there exist a simple $B_0(kN)$-module $S_0$ and a simple $B_0(k \tilde{H}_m)$-module $S_m$ such 
that 
$$S=\Inf_{N\times  \tilde{H}_m/1\times  \tilde{H}_m}^H(S_0)\otimes_k \Inf_{N\times  \tilde{H}_m/N\times 1}^H(S_m)\,.$$
By the above, $S_0$ is a projective $kN$-module (indeed $S{\downarrow}_N=(\dim_k S_m)S_0$), hence 
$$\Inf_{N\times  \tilde{H}_m/1\times  \tilde{H}_m}^H(S_0)$$
 is projective relatively to $ \tilde{H}_m$ and therefore so is $S$ seen as the above tensor product. This 
contradicts the fact that the vertices of $S$ are the Sylow $p$-subgroups of $H$.
Hence we conclude that $S$ must lie at the end of $\Gamma_s(S)$.
\end{proof}

\begin{prop}\label{prop:simplesZAinfinity}
Let  $G$ be a finite group  with $O_{p'}(G)=1$ and non-cyclic abelian Sylow $p$-subgroups. Assume 
moreover that one of Conditions (i), (ii), or (iii) of Theorem~\ref{thm:redsimples} is satisfied. Then every non-periodic  simple $B_0(kO^{p'}(G))$-module lies at the end of its AR-component.
\end{prop}

\begin{proof}
We have $O^{p'}(G)=Q$ or $O^{p'}(G)=Q \times H_1\times \cdots \times H_m$, where $Q$ is an abelian $p$-group and  $H_i$ is a non-abelian finite simple group with non-trivial Sylow $p$-subgroups for each $1\leq i\leq m$. \par
If (i) holds, that is $Q\neq 1$, then by Theorem~\ref{thm:end}(a), all simple $B_0(kO^{p'}(G))$-modules lie at 
the end of their AR-components.
Therefore, we assume for the rest of the proof that $Q=1$.\par
Next if  (ii) holds, that is $m\geq 2$, the claim follows from Lemma~\ref{lem:nonperiodicProduct}.\par

Finally if (iii) holds, that is $O^{p'}(G)=H_1$, then $H_1$ must have a non-cyclic Sylow $p$-subgroup, 
therefore all simple $B_0(kO^{p'}(G))$-modules lie at the end of their $AR$-components by assumption.
\end{proof}

\vspace{4mm}
\section{Reduction to  $O^{p'}(G)$} \label{sec:}

We continue  assuming  that $G$ is a finite group with non-cyclic abelian Sylow $p$-subgroups such that ${O_{p'}(G)=1}$, unless otherwise stated.
We now prove that an answer to Question~\ref{Ques:Q1} is  detected by restriction to the normal subgroup  
$O^{p'}(G)$ of $G$.\\

We set   $H :=O^{p'}(G)$, let $P\in \Syl_p(H)$ be a Sylow $p$-subgroup, and set $N:=HC_G(P)$. Moreover 
we set $B:=B_0(kG)$, $b:=B_0(kN)$ and $\tilde{b}:=B_0(kH)$.  Then $N$ is Dade's Group $G[\tilde{b}]$ and 
$N\trianglelefteq G$, see Lemma~\ref{G[b]}.\\

First of all Question~\ref{Ques:Q1} has an affirmative answer for the group  $N$ if and only if it has an 
affirmative answer for the group  $H$.

\begin{lem}\label{lem:alpdadered}
With the above notation, every simple $b$-module lies at the end of its AR-component if and only if every 
simple $\tilde{b}$-module lies at the end of its AR-component.
\end{lem}

\begin{proof}
By the Alperin-Dade Theorem \cite[Theorem]{Dade77}, the blocks $b$ and $\tilde{b}$ are isomorphic as $k$-algebras, hence  Morita equivalent. But for a simple module, lying at the end of its AR-component is a 
property preserved by Morita equivalence.
\end{proof}

\begin{prop}\label{prop:redO^p'}
If every simple $\tilde{b}$-module lies at the end of its AR-component, then every simple $B$-module lies at 
the end of its AR-component.
\end{prop}

\begin{proof}
Let $S$ be a simple $B$-module and let $T$ be a simple direct summand of $S{\downarrow}_{H}$. Then $T$ 
is periodic if and only if $S$ is. Therefore $\Gamma_s(S)~\cong~ \IZ A_{\infty}$ if and only if $
\Gamma_s(T)~\cong~ \IZ A_{\infty}$, and  $\Gamma_s(S)$ is an infinite tube  with tree class $A_{\infty}$ if 
and only if $\Gamma_s(T)$ is an infinite tube  with tree class $A_{\infty}$.\par
In case $\Gamma_s(S)~\cong~ \IZ A_{\infty}$, then  $S$ lies at the end of $\Gamma_s(S)$ if and only if $T$ 
lies at the end of $\Gamma_s(T)$ by \cite[Lemma 1.5]{KMU00}.\par
In case $\Gamma_s(S)$  is an infinite tube  with tree class $A_{\infty}$, then by Corollary~\ref{cor:tubesO^p'}, 
$H$ is a non-abelian finite simple group with non-cyclic abelian Sylow $p$-subgroups. Now, by Schreier's 
conjecture (now proven by the Classification of Finite Simple Groups, see 
\cite[Definition 2.1]{GLS2} 
 \cite[Theorem 7.1.1]{GLS3}), we know that $G/H$ is a solvable $p'$-subgroup of $\Out(H)$. 
Now by Lemma~\ref{lem:alpdadered}, we may assume $H=N$ and  by Lemma~\ref{G[b]}(c) we have 
$1_B=1_b$. Therefore Theorem~\ref{thm:solvablequotient} implies that  $S$ lies at the end of $\Gamma_s(S)$ because every simple $b$-module lies at the end of its AR-component.
\end{proof}

As a corollary, we obtain Theorem~\ref{thm:redsimples} of the Introduction.

\begin{proof}[Proof of Theorem~\ref{thm:redsimples}]
Let $G$ be a finite group with non-cyclic abelian Sylow $p$-sub\-groups. As $B_0(kG)$ and $B_0(kG/O_{p'}(G))$ are Morita equivalent, we may assume that $O_{p'}(G)~=~1$. Therefore, 
by Proposition~\ref{prop:redO^p'}, every simple $B_0(kG)$-module lies at the end of its AR-component  if 
every simple $kB_0(O^{p'})(G)$-module lies at the end of its AR-component.
Now if $B_0(G)$ contains a periodic simple module, then by Corollary~\ref{cor:tubesO^p'}  we must have that   
$O^{p'}(G)=H_1$  is a non-abelian finite simple group with non-cyclic abelian Sylow $p$-subgroups, then the 
claim holds by Corollary~\ref{cor:periodic}. Therefore we may assume that $B_0(kG)$, and hence 
$B_0(kO^{p'}(G))$, contains no periodic simple module. In this case, if one of Conditions (i),(ii), or (iii) holds, 
then the claim follows from Proposition~\ref{prop:simplesZAinfinity}.
\end{proof}

Now Corollary~\ref{thm:corredsimples} is a direct consequence of Theorem~\ref{thm:redsimples}.

\vspace{4mm}
\section{Principal $3$-blocks}\label{sec:p=3}

We now fix $p:=3$, and continue assuming that $G$ is a finite group with non-cyclic Sylow $3$-subgroups, 
so that $B_0(kG)$ is wild. We may also assume that $O_{3'}(G)=1$.\\

We start by investigating principal $3$-blocks of non-abelian finite simple groups with abelian defect group. To 
this aim, we recall that the list of non-abelian finite simple groups  with abelian Sylow $3$-subgroups is known 
by the classification of finite simple groups and was determined by Paul Fong (in an unpublished manuscript).

\begin{prop}[{}{\cite[Proposition 4.3]{KY10}}]\label{prop:simpleabdef}
If $G$ is a non-abelian finite simple group with non-cyclic abelian Sylow $3$-subgroup, then $G$ is one of:
\begin{enumerate}
  \item[\rm(i)] $\fA_7$, $\fA_8$, $M_{11}$, $M_{22}$, $M_{23}$, $HS$, $O'N$;
  \item[\rm(ii)]  $\PSL_3(q)$ for a prime power $q$ such that $3||(q-1)$;
  \item[\rm(iii)] $\PSU_3(q^2)$ for a prime power $q$ such that $3||(q+1)$;
  \item[\rm(iv)]  $\PSp_4(q)$ for a prime power $q$ such that $3|(q-1)$; 
  \item[\rm(v)]  $\PSp_4(q)$ for a prime power $q$ such that $q>2$ and $3|(q+1)$; 
  \item[\rm(vi)]  $\PSL_4(q)$ for a prime power $q$ such that $q>2$ and $3|(q+1)$; 
  \item[\rm(vii)]  $\PSU_4(q^2)$ for a prime power $q$ such that $3|(q-1)$; 
  \item[\rm(viii)]  $\PSL_5(q)$ for a prime power $q$ such that  $3|(q+1)$; 
  \item[\rm(ix)]  $\PSU_5(q^2)$ for a prime power $q$ such that $3|(q-1)$; or
  \item[\rm(x)]  $\PSL_2(3^n)$ for an integer $n\geq 2$.
\end{enumerate}
\end{prop}

\noindent As a consequence we obtain:

\begin{prop}\label{prop:B0nonabsimple}
If $G$ is a non-abelian finite simple group with  non-cyclic abelian Sylow $3$-subgroups, then every simple 
$B_0(kG)$-modules lies at the end of its component in $\Gamma_s(B_0(kG))$.
\end{prop}

\begin{proof}
Let $P\in\Syl_3(G)$, and set $N:=N_G(P)$ and  $B_0:=B_0(kG)$. We go through the list of groups in 
Proposition~\ref{prop:simpleabdef}.\par
In case (i), in all cases all simple $B_0$-modules lie at the end of their component in $\Gamma_s(B_0)$ by 
Theorem~\ref{thm:kawCartan}(b): indeed if $G$ is one of $\fA_8$, $M_{22}$ or $O'N$, then one checks from 
GAP \cite{GAP4} that the Cartan matrix of $B_0$ has no diagonal entry equal to $2$. If $G$ is one of $
\fA_7$, $M_{11}$, $M_{23}$, or $HS$, then one checks from GAP \cite{GAP4} that the Cartan matrix of 
$B_0$ does not have the shape of 
Theorem~\ref{thm:kawCartan}(b) either. \par
In case (ii), then the Cartan matrix of $B_0$ is computed in \cite[Table~2]{Kun00} and does not satisfy 
Theorem~\ref{thm:kawCartan}(b).\par
Next if $G$ is one of the groups listed in Proposition \ref{prop:simpleabdef}(iii),(iv),(vii), or (ix), then it is 
proven in \cite[Lemma 3.7]{KY10} that $B_0$ is Puig equivalent to $B_0(kN)$. But $N$ has a non-trivial 
normal Sylow $3$-subgroup, therefore all simple $B_0(kN)$-modules lie at the end of their components in $\Gamma_s(B_0(kN))$ by Theorem~\ref{thm:end}(a), and therefore so do the simple $B_0$-modules via the 
latter Puig (Morita) equivalence. \par
In case (v), the decomposition numbers of $B_0$  were computed by White and  Okuyama-Waki. If $q$ is 
even then we read from \cite[Table~II]{Whi95} that each column of the decomposition matrix of $B_0$ has at 
least~$3$ positive entries.  If $q$ is odd, then the decomposition matrix of $B_0$ is given in 
\cite[Theorem~4.2]{Whi90} up to two  parameters $\alpha$ and $\beta$. But \cite[Theorem~2.3]{OW98} 
proves that $\alpha\in\{1,2\}$. This is enough to see that each column of the decomposition matrix of $B_0$ 
has at least~$3$ positive entries. Therefore in both cases all the diagonal entries of the Cartan matrix of 
$B_0$ are at least $3$.  \par
In case (vi) and (viii), we proceed as follows. For $n\in \{4,5\}$ fixed,  we may regard $B_0(k\PSL_n(q))$ as 
the principal block of $\SL_n(q)$ as $3\nmid |Z(\SL_n(q))|$. 
Then we check that the Cartan matrix of  $B_0(k\GL_n(q))$ does not satisfy Theorem~\ref{thm:kawCartan}(b). To this end we use the information on the decomposition numbers of $B_0(k\GL_n(q))$ provided in 
\cite[Appendix~I]{James90}.  In both cases, it is enough to consider only the square submatrix  $\Delta_{n,0}$ 
of  the decomposition matrix of $B_0(k\GL_n(q))$ whose rows are indexed by the unipotent characters. Both 
in case $n=4$ and $n=5$, there are five modular characters in the principal block (using \cite{FS82}) and 
$$ \Delta_{4,0}=\left[\begin{array}{ccccccc}  (4)&   & 1  &   &   &  &  \\ (31) &   & 1  & 1  &   &  & \\  (2^2)&   &   
& 1  & 1  &  & \\  (21^2)&   &   1&   1& 1  & 1 & \\  (1^4)&   &  1 &   &   & 1& 1  \end{array}\right]\qquad 
\Delta_{5,0}=\left[\begin{array}{ccccccc}  (5)&   & 1  &   &   &  &   \\  (32)&   &   & 1  &   &  & \\  (31^2)&   &   1&   
1& 1  &  & \\  (2^21)&   &   &  1 &1   & 1&  \\  (1^5)&   &  1 &   &   1&  & 1 \end{array}\right]\,.$$
(See e.g. \cite[Proposition~3.1 and Proposition~4.1]{KM00}.)
It follows that the Cartan integers of $B_0(\GL_n(q))$ have lower bounds given by the entries of the following 
matrices:
$$^{T}\!\Delta_{4,0}\Delta_{4,0}=\left[\begin{array}{ccccc}  4 & 2  &  1 &   2& 1  \\ 2  & 3  & 2  & 1  & 0  \\  1&  2 
&   2&   1& 0  \\ 2 & 1  &  1 & 2  & 1   \\  1&  0 &  0 & 1  &1  \end{array}\right]
\qquad 
^{T}\!\Delta_{5,0}\Delta_{5,0}=\left[\begin{array}{ccccc} 3 &  1 & 2  & 0  & 1  \\  1 & 3  & 2  & 1  &  0 \\ 2 & 2   & 
3  & 1  &  1 \\ 0 &  1 &  1 & 1  &  0 \\ 1 & 0  & 1  & 0  & 1 \end{array}\right]
$$
Therefore the Cartan matrix of  $B_0(k\GL_n(q))$ cannot satisfy Theorem~\ref{thm:kawCartan}(b), and we 
conclude that all simple $B_0(k\GL_n(q))$-modules lie at the end of their AR-components.
Now, from the known values of the unipotent characters of $\GL_n(q)$, we easily check that the dimension of 
the simple modules in $B_0(k\GL_n(q))$ are prime to $3$, hence they cannot be periodic by \cite{Car79}, as 
$3^{(a-1)}$ must divide the dimension of any simple periodic module, 
where $a:=$ the $p$-rank of the group, but in our case $a\geq 2$.
Therefore every simple $B_0(k\SL_n(q))$-module lies at the end of its AR-component by \cite[Lemma~1.5]{KMU00}.\par
Finally, if $G=\PSL_2(3^n)$ for some integer $n\geq 2$, then the claim follows  from Theorem~\ref{thm:end}(c) as $G$ is a finite simple group of Lie type in defining characteristic.
\end{proof}

As a corollary we obtain Theorem~\ref{thm:p=3} of the Introduction.

\begin{proof}[Proof of Theorem~\ref{thm:p=3}]
The claim now follows from Corollary~\ref{thm:corredsimples} together with 
Proposition~\ref{prop:simpleabdef}.
\end{proof}



\vspace{1cm}

\noindent 
{\bf Acknowledgements.}
{\small
Both authors gratefully acknowledge financial support and the hospitality provided by the  Centre 
Interfacultaire Bernoulli (CIB) of the \'{E}cole Polytechnique F\'{e}d\'{e}rale de Lausanne during the writing 
period of this article.
The first author  is grateful to the hospitality of the Department of Mathematics in TU Kaiserslautern. }

%


\begin{thebibliography}{KMU00}

\bibitem[Alp76]{Alp76}
J.L.~Alperin, \emph{Isomorphic blocks},
J.~Algebra \textbf{43} (1976), 694--698.

\bibitem[BU01]{BU01}
C.~Bessenrodt and K.~Uno, \emph{Character relations and simple modules in the
  {A}uslander-{R}eiten graph of the symmetric and alternating groups and their
  covering groups}, 
  Algebr. Represent. Theory \textbf{4} (2001), 445--468.

\bibitem[Car79]{Car79}
J.F.~Carlson, \emph{The dimensions of periodic modules over modular group algebras},
 Illinois J. Math. \textbf{23} (1979), 295--306.

\bibitem[Dad73]{Dade73}
E.C.~Dade, \emph{Block extensions},
Illinois J.~Math. \textbf{17} (1973), 198--272. 

\bibitem[Dad77]{Dade77}
E.C.~Dade, \emph{Remarks on isomorphic blocks},
J.~Algebra \textbf{45} (1977), 254--258. 

  
\bibitem[Erd90]{ERDMANNbook}
K.~Erdmann, \emph{{Blocks of tame representation type and related algebras.}},
  {Lecture Notes in Mathematics, 1428. Berlin etc.: Springer-Verlag}, 1990.  

\bibitem[Erd95]{Erd95}
K.~Erdmann, \emph{On {A}uslander-{R}eiten components for group algebras}, J.
  Pure Appl. Algebra \textbf{104} (1995), no.~2, 149--160.

\bibitem[FH93]{FH}
P.~Fong and M.E.~Harris, \emph{On perfect isometries and isotypies in finite
  groups}, Invent. Math. \textbf{114} (1993), 139--191.
  
\bibitem[FS82]{FS82}
P.~Fong and B.~Srinivasan, \emph{The blocks of finite general linear and
  unitary groups}, Invent. Math. \textbf{69} (1982), 109--153.  

\bibitem[GAP13]{GAP4}
The GAP~Group, \emph{{GAP -- Groups, Algorithms, and Programming, Version
  4.6.5}}, 2013.


\bibitem[GLS3]{GLS3}
D.~Gorenstein, R.~Lyons and R.~Solomon,
\emph{The Classification of the Finite Simple Groups, Number 3},
Math.~Survey Monographs, Amer.~Math.~Soc., 1998.


\bibitem[Jam90]{James90}
G.~James, \emph{The decomposition matrices of {${\rm GL}_n(q)$} for {$n\le
  10$}}, Proc. London Math. Soc. (3) \textbf{60} (1990), 225--265.

\bibitem[Kaw97]{Kaw97}
S.~Kawata, \emph{On {A}uslander-{R}eiten components and simple modules for
  finite group algebras}, Osaka J. Math. \textbf{34} (1997), 681--688.
  
\bibitem[KMU00]{KMU00}
S.~Kawata, G.~O. Michler, and K.~Uno, \emph{On simple modules in the
  {A}uslander-{R}eiten components of finite groups}, Math. Z. \textbf{234}
  (2000), 375--398.

\bibitem[KMU01]{KMU01}
\bysame, \emph{On {A}uslander-{R}eiten components and simple modules for finite
  groups of {L}ie type}, Osaka J. Math. \textbf{38} (2001), 21--26.
  
\bibitem[Kn{\"o}79]{Knorr79}
R.~Kn{\"o}rr, \emph{On the vertices of irreducible modules}, Ann. of Math. (2)
  \textbf{110} (1979), 487--499.  

\bibitem[Kos04]{KoCo}
S.~Koshitani, \emph{Corrigendum: ``{C}onjectures of {D}onovan and {P}uig for
  principal 3-blocks with abelian defect groups'' [{C}omm. {A}lgebra {\bf 31}
  (2003), no. 5, 2229--2243; mr1976275]}, Comm. Algebra \textbf{32} (2004),
  391--393.

\bibitem[KK02]{KK02}
S.~Koshitani and N.~Kunugi, 
\emph{Brou\'e's conjecture holds for principal
$3$-blocks with elementary abelian defect group of order $9$},
J. Algebra \textbf{248} (2002), 575--604.

 \bibitem[KM00]{KM00}
S.~Koshitani and H.~Miyachi, \emph{The principal 3-blocks of four- and
  five-dimensional projective special linear groups in non-defining
  characteristic}, J. Algebra \textbf{226} (2000), 788--806.   

\bibitem[KY10]{KY10}
S.~Koshitani and Y.~Yoshii, \emph{Eigenvalues of {C}artan matrices of principal
  3-blocks of finite groups with abelian {S}ylow 3-subgroups}, J. Algebra
  \textbf{324} (2010), 1985--1993.

\bibitem[Kue90]{Kue90}
B.~K{\"u}lshammer, 
\emph{Morita equivalent blocks in Clifford theory of finite groups},
Ast\'erisque \textbf{181--182}  (1990), 181--182.

\bibitem[Kue95]{Kue95}
B.~K{\"u}lshammer, 
\emph{Donovan's conjecture, crossed products and algebraic group actions},
Israel J.~Math. \textbf{92} (1995), 295--306.


\bibitem[Kun00]{Kun00}
N.~Kunugi, \emph{Morita equivalent 3-blocks of the 3-dimensional projective
  special linear groups}, Proc. London Math. Soc. (3) \textbf{80} (2000),
  575--589.

\bibitem[Lan83]{Lan83}
P.~Landrock,
\emph{Finite Group Algebras and their Modules}, 
London Math.~Soc.~Lecture Note Series \textbf{84},
Cambridge University Press, 
Cambridge, 1983.

\bibitem[NT88]{NagaoTsushima}
H.~Nagao and Y.~Tsushima,
\emph{Representations of Finite Groups}, Academic Press, New York,
1988.

\bibitem[OU94]{OkUn94}
T.~Okuyama and K.~Uno, \emph{On the vertices of modules in the
  {A}uslander-{R}eiten quiver. {II}}, Math. Z. \textbf{217} (1994), no.~1,
  121--141.
   
\bibitem[OW98]{OW98}
T.~Okuyama and K.~Waki, \emph{Decomposition numbers of {${\rm Sp}(4,q)$}}, J.
  Algebra \textbf{199} (1998), 544--555.   


\bibitem[Vil59]{Vill59}
O.E.~Villamayor,
\emph{On the semisimplicity of group algebras II},
Proc.~Amer.~Math.~Soc. \textbf{10} (1959), 21--27.

\bibitem[Whi90]{Whi90}
D.L.~White, \emph{Decomposition numbers of {${\rm Sp}(4,q)$} for primes
  dividing {$q\pm 1$}}, J. Algebra \textbf{132} (1990), 488--500.

\bibitem[Whi95]{Whi95}
\bysame, \emph{Decomposition numbers of {${\rm Sp}_4(2^a)$} in odd
  characteristics}, J. Algebra \textbf{177} (1995), 264--276.
  
  
\end{thebibliography}


\end{document}